\documentclass[12pt,reqno]{amsart}

\usepackage{fancyhdr}
\usepackage{psfrag}
\usepackage{graphicx}
\usepackage{amsmath}
\usepackage{amsfonts}
\usepackage{amssymb}		
\usepackage{amsthm}
\usepackage{mathabx}
\usepackage[indent]{caption}
\usepackage{indentfirst}
\usepackage{float}
\usepackage{latexsym}
\usepackage{graphics}
\usepackage{verbatim}

\newcommand{\Q}{{\mathbb Q}}
\newcommand{\C}{{\mathbb C}}

\newcommand{\Z}{{\mathbb Z}}

\newcommand{\N}{{N}}
\newcommand{\FF}{{\mathbb F}}

\newcounter{dummy}\numberwithin{dummy}{section}

\theoremstyle{plain}
\newtheorem{question}[dummy]{Question}
\newtheorem{theorem}[dummy]{Theorem}
\newtheorem{lemma}[dummy]{Lemma}
\newtheorem{prop}[dummy]{Proposition}

\newtheorem{corollary}[dummy]{Corollary}
\newtheorem{conjecture}[dummy]{Conjecture}

\theoremstyle{definition}

\newtheorem{remark}[dummy]{Remark}

\begin{document}

\author{Manisha Kulkarni}
\address{International Institute of Information Technology Bangalore,
Hosur Road, Bangalore, INDIA 560100.\\
email: manisha.shreesh@gmail.com} 
\author{Vijay M. Patankar}
\address{School of Physical Sciences, Jawaharlal Nehru University, New
  Delhi, INDIA 110067.\\
email: vijaypatankar@gmail.com}  
\author{C. S. Rajan}
\address{School of Mathematics, Tata Institute of Fundamental
  Research, Dr. Homi Bhabha Road, Colaba, Bombay, INDIA 400005.\\
email: rajan@math.tifr.res.in}





\title[Potentially equivalent Galois representations]
{Locally potentially equivalent two dimensional Galois representations and Frobenius fields of elliptic curves}

\begin{abstract}

We show that a two dimensional $\ell $-adic representation of the absolute Galois group of a number field which is locally potentially equivalent to a $GL(2)$-$\ell$-adic representation $\rho$ at a set of places of
$K$ of positive upper density is potentially equivalent to $\rho$.

For an elliptic curver \( E \) defined over a number field \( K \) and for a place \( v \) of \( K \) of good reduction for \( E \), let \( F(E; v) \) denote the Frobenius field of \( E \) at \( v \), given
by the splitting field of the characteristic polynomial of the Frobenius automorphism at \( v \) acting on the Tate module of \( E \).

As an application, suppose \( E_1 \) and \( E_2 \) defined over a number field \( K \), with at least one of them without complex
multiplication. We prove that the set of  places \( v \) of \( K \) of good reduction such that the corresponding Frobenius fields are equal
has positive upper density if and only if  \( E_1 \) and \( E_2 \) are isogenous over some extension of \( K \).

For an elliptic curve $E$ defined over a number field \( K \), we show that the set of finite places of \( K \) such that  the Frobenius field \( F(E,v) \) at $v$
equals a fixed imaginary quadratic field \( F \) has positive upper density if and only if \( E \) has complex multiplication by \( F \).
\end{abstract}

\maketitle  

\section{Introduction}

Let \(E\) be an elliptic curve defined over  a number field \(K\).
Let $G_K={\rm Gal}(\bar{K}/K)$ denote
the absolute Galois group over $K$ of a separable closure $\bar{K}$ of
$K$. The Galois group $G_K$ acts in a natural manner on $E(\bar{K})$.
For a rational prime $\ell$, the Tate module
$T_\ell(E):=\varprojlim_n E[\ell^n]$ is the $G_K$-module obtained
as a projective  limit of the  $G_K$-modules $E[\ell^n]$ of
$\ell^n$-torsion points of $E$ over $\bar{K}$.  Let
$V_{\ell}(E)=T_\ell (E)\otimes_{\Z_l}\Q_l$. The Tate
module is of rank $2$ over the ring of $\ell$-adic integers $\Z_\ell$,
and we have a  continuous $\ell$-adic representation
$\rho_{E,\ell}:G_K\to GL_2(\Q_\ell)$.

Let \( \Sigma_K \) denote the set of finite places of \( K \)
and $\Sigma_r\subset
\Sigma_K$ a finite set of places containing the finite places of
bad reduction for $E$. The Galois module
$V_\ell(E)$ is unramified at the finite places of $K$ outside
$\Sigma_{r,\ell}=\Sigma_r\cup\{v|\ell\}$.
For $v\not\in \Sigma_{r,\ell}$, let $\rho_{\ell}(\sigma_v)$ denote
the Frobenius conjugacy class contained in $GL_2(\Q_\ell)$.
The representations $\rho_\ell$ form a compatible system of $\ell$-adic
representations, in that the characteristic polynomial \(
\phi_v (t) \) of  $\rho_{\ell}(\sigma_v)$ is independent of \(  \ell \) and its
coefficients  are integral.
Thus, \( \phi_v (t) := t^2 - a_v (E) t + \N v \), with
$a_v(E), ~Nv\in \Z$. Here,  $Nv$ is the cardinality of the
residue field  $ k_v := {\mathcal
O}_K/{\mathfrak p}_v$, where ${\mathcal O}_K$ is the ring of integers
of $K$, and ${\mathfrak p}_v$ is the prime ideal of ${\mathcal O}_K$
corresponding to $v$.

Define the
\emph{Frobenius field} \( F(E, v) \) of \( E \) at \( v \)
as the splitting field of
\( \phi_v (t) \) over \( \Q \). Thus,
\( F( E, v) = \Q ( \pi_v ) = \Q ( \sqrt{ a_v(E)^2 - 4 \N v } ) \),  where
\( \pi_v \) is a root of \( \phi_v (t) \). The Hasse bound
$|a_v(E)|\leq 2\sqrt{Nv}$ implies that \( F (E, v ) \) is either \( \Q
\) or an imaginary quadratic field.

As an application of and motivation for the main theorem (Theorem
\ref{poteq}), we have the following multiplicity-one type
theorem under the assumption that the set
of places \( v \) for which  the Frobenius fields coincide has
positive upper density as defined below.

\begin{theorem}\label{Mult-one-FF} Let \( E_1 \) and \( E_2 \) be two
elliptic curves over a number field \( K \). Let $\Sigma_r$ be a finite
subset of the set  \( \Sigma_K \) of finite places of  \(K\)
containing the places of bad reduction of $E_1$ and $E_2$.  Assume
that at least one of the elliptic curves is without complex
multiplication. Let
\[ S (E_1, E_2 ) := \{ v \in \Sigma_K ~\backslash ~\Sigma_r~\mid ~ F (E_1, v) = F ( E_2, v)
\}.
\] Then, \( E_1 \) and \( E_2 \) are isogenous over a finite extension
of \( K \) if and only if \( S ( E_1 , E_2 ) \) has  positive upper
density.
\end{theorem}

We recall the notion of upper density:  given a set \( S \subset \Sigma_K  \) of finite places of
\( K \), recall that the \emph{upper density} \( ud(S) \) of \( S \)
is defined as,
\[
ud(S) := {\rm  lim sup}_{x \rightarrow \infty} \frac{\#\{v\in
\Sigma_K ~ \mid ~ Nv\leq x,~v\in S\}}{ \#\{v\in \Sigma_K ~ \mid ~
Nv\leq x\}}.
\]
The advantage of working with upper density is that it always exists,
whereas the (naive) density is defined as the limit (if it exists) of
the above expression:
\[
  d(S) := {\rm  lim}_{x \rightarrow \infty} \frac{\#\{v\in \Sigma_K ~
  \mid ~ Nv\leq x,~v\in S\}}{ \#\{v\in \Sigma_K ~ \mid ~ Nv\leq x\}}.
\]

Using Theorem \ref{Mult-one-FF}, we prove:
\begin{theorem}\label{FF-with-CM}
Let \( E \) be an elliptic curve
over a number field \( K \). Let \( F \) be an imaginary quadratic
field. Let $\Sigma_r$ be a finite subset of the set  \( \Sigma_K \)
containing the places of bad reduction of $E$.
Let \( S(E, F ) := \{ v \in \Sigma_K ~\backslash ~\Sigma_r~| ~F(E, v) = F \} \).
Then, \( S( E, F ) \) has positive upper density if and only if \( E \) has complex multiplication by \( F \).
\end{theorem}

As a consequence, we prove:

\begin{corollary}
Let \( E \) be an elliptic curve over a number field \( K
\). Let \( F(E) \) be the compositum of the Frobenius fields \( F(E,v)
\) as \( v \) varies over places of good ordinary reduction for \( E
\). Then, \( F(E ) \) is a number field if and only if \( E \) is an
elliptic curve with complex multiplication.
\end{corollary}

The above corollary also follows from a set of exercises  in Serre's
book \cite{S1} Chapter IV, pages 13-14. Thus, Theorem
(\ref{FF-with-CM}) can be considered as a strengthening of this
corollary.
\medskip\par

\begin{remark}
In an email communication, Serre has given an alternate proof of Theorem \ref{Mult-one-FF}. We briefly explain his proof in Remark \ref{Serre-proof}.
\end{remark}

\subsection{Potentially equivalent Galois representations}
In an earlier version of this paper, we used the results of (\cite{PR}) on global equivalence
of locally potentially equivalent $\ell$-adic representations to prove
Theorem \ref{Mult-one-FF}.  It was pointed out
by J.-P. Serre that the proof of Theorem 2.1 given in \cite{PR}
is erroneous. We salvage this by proving a version of Theorem 2.1 of \cite{PR} for \( n = 2 \).

For any place $v$ of $K$, let $K_v$ denote the completion of $K$ at
$v$, and $G_{K_v}$ the corresponding local Galois group. Choosing a
place $w$ of $\bar{K}$ lying above $v$, allows us to identity
$G_{K_v}$ with the decomposition subgroup $D_w$ of $G_K$. As $w$
varies this gives a conjugacy class of subgroups of $G_K$. Given a
representation $\rho: G_K\to GL_n(F)$, for some field \( F \), define the localization (or
the local component) \( \rho_{v} \) of $\rho$ at $v$, to be the
representation of $G_{K_v}$ obtained by restricting $\rho$ to a
decomposition subgroup at \( v \). This is well defined upto isomorphism.

At a place $v $ of $K$ where
$\rho$ is unramified, let $\rho(\sigma_v)$ denote the Frobenius conjugacy
class in the quotient group $G_K/{\rm Ker}(\rho)$. By an abuse of
notation, we will also continue to denote by \( \rho(\sigma_v) \) an element
in the associated conjugacy class.

Define  the algebraic
monodromy  group $G$ attached to \( \rho \) to be
the smallest algebraic subgroup $G$ of $GL_n$ defined over $F$ such that
$\rho(G_K)\subset G(F)$.

We now give a version of Theorem 2.1 of (\cite{PR}) for two
dimensional Galois representations.

\begin{theorem}\label{poteq}
Suppose $\rho_i: G_K \to GL_2(F), ~i=1,2$ are two continuous
semisimple \( \ell \)-adic representations of the absolute
Galois group $G_K$ of a global field $K$ unramified outside a finite set of
places of $K$, where \( F \) is a non-archimedean
local field of characteristic zero
and residue characteristic \( \ell \) coprime to the
characteristic of \( K \).

Suppose there exists a set $T$ of finite places of $K$ of
positive upper density
such that for every $v\in T$, $\rho_{1,v} $ and $\rho_{2,v}$ are potentially equivalent.

Assume that the algebraic monodromy group $G_{1}$ attached to the representation
$\rho_{1}$ is isomorphic to $GL_2$, and that the determinant
characters of $\rho_1$ and $\rho_2$ are equal.

Then, $\rho_{1}$ and $\rho_{2}$ are potentially equivalent,
viz. there exists a finite extension $L$ of $K$ such that
\[
\rho_{1} \mid_{G_L} \simeq \rho_{2} \mid_{G_L} .
\]
\end{theorem}

The proof of Theorem \ref{poteq} uses an
``analytic and algebraic continuation'' of the Galois groups,
involving specializing to appropriate elements
in the algebraic Galois monodromy groups.

The proof of the following theorem (\cite[Theorem 3.1]{PR}) was given
as a consequence of (\cite[Theorem 2.1]{PR}).
Here we give a proof, following closely the original argument, but
removing the dependence on (\cite[Theorem 2.1]{PR}).

\begin{theorem}\label{characterpowers}
Suppose $\rho_i: G_K \to GL_n(F), ~i=1,2$ are two continuous
semisimple \( \ell \)-adic representations of the absolute
Galois group $G_K$ of a global field $K$ unramified outside a finite set of
places of $K$, where \( F \) is a non-archimedean
local field of characteristic zero
and residue characteristic \( \ell \) coprime to the
characteristic of \( K \).

Assume that there exists a set $T$ of finite
places of $K$  not containing the ramified places of $\rho_1 \times
\rho_2$ and the places of $K$ lying above $\ell$,
such that for every $v\in T$, there exists non-zero
integers  $m_v> 0 $ satisfying the following:
\[
    \chi_{\rho_1}(\sigma_v)^{m_v} =\chi_{\rho_2}(\sigma_v)^{m_v} ,
\]
where $\chi_{\rho_1}(\sigma_v)$ (resp. $\chi_{\rho_1}(\sigma_v)$) is
the trace of the image of the Frobenius conjugacy class
${\rho_1}(\sigma_v)$ (resp. ${\rho_2}(\sigma_v)$).

Suppose the upper density \( ud (T) \) of \( T \) is positive, and
the Zariski closure of
\( \rho_1 (G_K) \) is a connected algebraic group.

Then, $\rho_1$ and $\rho_2$ are potentially equivalent, viz. there
exists a finite extension $L$ of $K$ such that
\[
\rho_1 \mid_{G_L} \simeq \rho_2 \mid_{G_L} .
\]
\end{theorem}

\section{Proof of Theorem \ref{characterpowers}}

In this section, we give a proof of Theorem \ref{characterpowers}. The
argument is essentially the one given in (\cite{PR}), but we remove the dependence on (\cite[Theorem 2.1]{PR}).

\subsection{An algebraic Chebotarev density theorem}\label{Alg-Cheb-section}

We recall an algebraic version of the Chebotarev density theorem of
(\cite{R1}, see also \cite{S3}). Here $F$ stands for a non-archimedean
local field of characteristic $0$.

\begin{theorem}\label{Algebraic-Chebotarev} \cite[Theorem 3]{R1}~~
Let \( M \) be an algebraic group defined over \( F \). Suppose
\[
\rho : G_K \rightarrow M(F)
\]
is a continuous representation unramified
outside a finite set of places of $K$. Let $G$  be the Zariski closure
inside $M$ of the image $\rho(G_K)$, and $G^0$ be the connected
component of identity of $G$. Let $\Phi=G/G^0$ be the group of
connected components of $G$.

Suppose  \( X \) is a closed subscheme of \( M\) defined over
\( F \) and stable under the adjoint action of \( M \)
on itself. Let
\[
C := X(F) \cap \rho ( G_K ) .
\]
Let \( \Sigma_u \) denote the set of finite places of \( K \)
at which \( \rho \) is unramified, and $\rho(\sigma_v)$ denote the Frobenius
conjugacy class in $M(F)$ for $v\in \Sigma_u$.
Then the set
\[
    S : =  \{ v \in \Sigma_u  ~| ~ \rho ( \sigma_v ) \subset C \}.
\]
has a density given by
\[
   d(S) = \frac{ | \Psi | }{ | \Phi | },
\]
where $\Psi $ is the set of those $\phi \in \Phi$ such that the
corresponding connected component $G^\phi$ of $G$ is contained in $X $.
\end{theorem}

\subsection{Proof of Theorem  \ref{characterpowers}}

If $\chi_{\rho_1}(\sigma_v)$ vanishes, then the hypothesis holds for any
integer $m_v$. On the other hand, if $\chi_{\rho_1}(\sigma_v)$ is non-zero, then $\chi_{\rho_1}(\sigma_v)$  and
$\chi_{\rho_2}(\sigma_v)$ differ by a root of unity belonging to
$F$. Since the group of roots of unity in the non-archimedean local
field $F$ is finite, there is an integer $m$ independent of $v$,
such  that for $v\in T$,
\[
	\chi_{\rho_1}(\sigma_v)^{m} =\chi_{\rho_2}(\sigma_v)^{m}
\]
Let
\[
  X^m := \{ ( g_1 , g_2 ) \in GL_n \times GL_n ~| ~ \mbox{Trace} (g_1)^m  = \mbox{Trace} (g_2)^m  ~ \}.
\]
\( X^m \) is a Zariski closed subvariety of \( GL_n \times GL_n \)
invariant under conjugation. Let $G$ be the Zariski closure in
$GL_n\times GL_n$ of the image $\rho_1\times \rho_2(G_K)$.
By Theorem \ref{Algebraic-Chebotarev}, the density condition on $T$ implies the existence a connected
component $G^{\phi}$ of $G$ contained inside $X^m$.  Since $G_1$ is
assumed to be connected, the projection map from $G^{\phi}$ to $G_1$
is surjective. Hence there is an element of the form $(1,y)\in
G^{\phi}(\overline{F})$.

Since $G$ is reductive, by working over the complex numbers
and with a maximal compact subgroup $J$ of $G$, we can assume that
there is an element of the form $(1,y)\in J^{\phi}\cap X^m$. Since the
only elements in an unitary group $U(n)$ with the absolute value of
it's trace being precisely $n$ are scalar matrices $\zeta I_n$ with
$|\zeta|=1$, we conclude that $y$ is of the form $\zeta I_n$ for
$\zeta$ a $m$-th root of unity.

We can write the connected component $G^{\phi}= G^0.(1,\zeta~I_n)$. In
particular, every element $(u_1, u_2)\in G^0$, the identity  component
of $G$, can be written as
\[
   (u_1, u_2)=(z_1, \zeta^{-1}z_2),
\]
where $ (z_1, z_2) \in G^{\phi} \cap X^m $.
Since $\zeta$ is a $m$-th root of unity, we have
\[
  \mbox{Trace}(u_1 )^m = \mbox{Trace}(z_1 )^m =\mbox{Trace}(z_2)^m
    = \mbox{Trace}( \zeta^{-1} z_2 )^m  =\mbox{Trace}(u_2^m).
\]
Hence $ G^0 \subset X^m$. Let $p_i, ~i=1,~2$ be the two
projections from $G^0$ to $GL(n)$.  The statement $G^0\subset X^m$
can be reformulated as saying that
\[
	\chi_{p_1}^m=\chi_{p_2}^m,
\]
restricted to $G^0$, where \( \chi_{p_1} \) and \( \chi_{p_2} \)
are the characters associated to \( p_1 \) and \( p_2 \)
respectively.

We now argue as in (\cite{R2}).
The characters $\chi_{p_1}$ and $\chi_{p_2}$
differ by an $m$-th root of unity. Since the characters
are equal at identity, they are equal on a connected
neighbourhood of identity. Since a neighbourhood of identity is
Zariski dense in a connected algebraic group, and the
characters are regular functions on
the group, it follows that the characters are equal on $G^0$.
Thus it follows that the representations $p_1$ and $p_2$ are equivalent
restricted to $G^0$. This proves that $\rho_1$ and $\rho_2$ are
potentially equivalent.

\begin{remark} \label{Unitary-Monodromy-at-Infinite-place}
Working with unitary groups and specializing to the identity element,
can be considered as a proof involving analytic continuation of
Galois monodromy at the infinite place.
\end{remark}

\section{Proof of Theorem \ref{poteq}}

In this section we give a proof of Theorem \ref{poteq}. We start with
the following lemma about semi-simple algebraic groups.

\subsection{A lemma on algebraic groups}

\begin{lemma} \label{surj-to-SL2}
Let \( G \) be a  connected reductive algebraic group defined over a
field $F$ of characteristic zero.
 Let \( p : G \rightarrow GL_2 \) be a surjective homomorphism
defined over $F$.  Then, \( p ( G(F) ) \)
contains \( SL_2 (F) \).
\end{lemma}

\begin{proof}
The induced map \( p \)  from the derived group \( G^{d} \) of $G$ to the derived subgroup
 \( SL_2 \) of $GL_2$ is a surjective homomorphism
defined over \( F \).

Since \( G^{d} \) is a connected semi-simple algebraic group over \( F
\),
there exists a surjective homomorphism
\( \prod_i G_i \to G^d\) with finite kernel defined over \( F \),
where each \(  G_i \) is a connected,  simply connected,
simple algebraic groups defined over $F$  (\cite[Theorem 22.10]{B}).

This gives a surjective homomorphism \( \psi \) from
\( \prod_i G_i \) to \( SL_2 \) over \( F \). Since \( SL_2 \) is simple,
it follows that for each \( i \), \( \psi |_{G_i} : G_i \rightarrow
SL_2 \) is either trivial or an isogeny of algebraic groups. In the
latter case, \( G_i \) is either a form of \( SL_2 \) or \( PSL_2 \).
 Since \( SL_2 \) is simply connected,
\( G_i \) is in fact a form of \( SL_2 \) over \( F \).
In other words, the induced map
\( \psi : G_i \rightarrow SL_2 \) is an isomorphism over \( \bar{F} \). However,
since \( \psi \) is defined over \( F \) itself, this proves that \(
\psi : G_i \rightarrow SL_2 \) is an isomorphism over \( F \), proving
the lemma.
\end{proof}

\subsection{An arithmetic lemma}
\begin{lemma}\label{arith}
 Let $F$ be a non-archimedean local field of
  characteristic zero and residue characteristic $l$. Suppose $d, ~a$
  are non-zero elements in the ring of integers ${\mathcal O}$ of
  $F$.
Then there exists $x\in F$ such that
$d-ax^2$ is not a square in $F$.
\end{lemma}
\begin{proof}
Suppose $d-ax^2$ is a square in $F$ for any value of $x\in F$.
Specializing  \( x = 0 \) it follows that \( d =
b^2 \) for some  \( b \neq 0 \in F \). Writing $x=y/z$, we get
\[b^2-ax^2= ((bz)^2-ay^2)/z^2.\]
It follows that the homogenous form $z^2-ay^2$ is a square in $F$ for
any $y\in F, ~z\in F^*$.

The form  $z^2-ay^2$ can be considered as the norm form from the
quadratic algebra $F(\sqrt{a})$ to $F$. From the multiplicativity of norms,
\[ (z_1^2-ay_1^2)(z_2^2-ay_2^2)=
(z_1z_2+ay_1y_2)^2-a(z_1y_2+z_2y_1)^2\]
it follows upon equating $z_1z_2+ay_1y_2=0$, that $-a$ is a square in
$F$.  The form $z^2-ay^2$ is equivalent to the norm form $z^2+y^2$
from the quadratic
algebra $F(\sqrt{-1})$ to $F$, and is a square in $F$ for any $z,
~y\in F$.

If   $\sqrt{-1}\in F$, then $F(\sqrt{-1})\simeq F\times F$, and the
norm form is equivalent to the product form $(z,y)\mapsto zy$, and
is surjective onto $F$. This implies that every element of $F$ is a
square, and yields a contradiction.

If $\sqrt{-1}\not\in F$, then the image of the non-zero elements of
the field  $F(\sqrt{-1})$ by the norm map is a subgroup of index $2$
in $F^*$ by local class field theory. The hypothesis implies that this
is contained in the group $(F^*)^2$ which is of index at least $4$
since $F$ is a non-archimedean, local field of characteristic
zero. This is a contradiction and establishes the lemma.
\end{proof}

\subsection{Proof of Theorem \ref{poteq}}
The non-semisimple elements in $GL_2$ are contained inside
a proper Zariski closed set
given by the vanishing of the discriminant of its characteristic
polynomial. By Theorem
\ref{Algebraic-Chebotarev}, it follows
that at a set of places of density one, the Frobenius conjugacy classes
\( \rho_{1} ( \sigma_v ) \) are semisimple. In particular, we can
assume by going to a subset of $T$ (denoted again by $T$)
 with the same upper density,
that for $v\in T$,  \( \rho_{1} (\sigma_v) \) is
semisimple.

The eigenvalues of $\rho_{1}(\sigma_v)$ and \( \rho_{2}(\sigma_v) \) lie in quadratic
extensions of $F$,  say $F_1 (v)$ and respectively,
\( F_2 (v) \). Let \( F(v) \) be the compositum of \( F_1 (v) \) and
\( F_2 (v) \). Thus \( F(v) \) is a Galois extension of $F$ and
contained in   a biquadratic extension of \( F\).

By hypothesis, at a place $v\in T$, \( \rho_{1}(\sigma_v)^{n_v}=
\rho_{2}(\sigma_v)^{n_v} \).

For \( v \in T\), let \( \pi_{1,v} , ~\overline{\pi}_{1,v} \) and  \( \pi_{2,v} ,
~\overline{\pi}_{2,v} \) be respectively the roots of the
characteristic polynomials of \( \rho_{1}(\sigma_v)\) and
\(\rho_{2}(\sigma_v)  \).
Upto reordering, we have $\pi_{2,v}=u\pi_{1,v}$ and \( \overline{ \pi}_{2,v}  =
\bar{u } \overline{\pi}_{1,v}  \)
for some roots of unity $u, ~\overline{u} \in F(v)
$.

Since there are only finitely many quadratic extensions of $F$, the
collection of fields $F(v)$ as $v$ varies lie in a fixed local
field $F'$.  In particular, the group of roots of unity $\mu_{F'}$ belonging
to $F'$ is finite. For  roots of unity  $u, ~\overline{u}\in F'$, let
\[
   T_u := \{ v \in T ~|~ \pi_{2,v}= u \pi_{1,v}\quad \mbox{and}\quad
\overline{ \pi}_{2,v}  = \bar{u } \overline{\pi}_{1,v} \}.
\]

Since, the upper density of \( T\) is positive, it follows that
\(  T_u \) has  positive upper density for some $u\in \mu_{F'}$.
We consider two cases.

\subsubsection{\( u = \overline{u} \)}
Let $m$ be the order of $u$. In this case, we obtain
\[
   T_{u} := \{ v \in T ~|~ Tr ( \rho_{1, \ell} (\sigma_v)  )^m
   = Tr ( \rho_{2,\ell} (\sigma_v) )^m \}.
\]
Suppose $T_u$ has positive upper density.
It follows from Theorem \ref{characterpowers}
that \( \rho_1 \) and \( \rho_2 \) are potentially equivalent.

\subsubsection{\( u \neq \overline{u}\)}
We have,
\[
  \left( \begin{array}{cc} 1 & 1  \\ u & \bar{u} \end{array} \right)
  \left( \begin{array}{c} \pi_{1,v} \\ \overline{ \pi}_{1,v}  \end{array} \right)
  =  \left( \begin{array}{c} t_1 \\ t_2 \end{array} \right) ,
\]
where \( t_1 := Tr( \rho_{1} ( \sigma_v ) ) \) and
\( t_2 := Tr( \rho_{2} (\sigma_v ) ) \). The determinant of this matrix is \( \bar{u}-u \neq 0
\), and we get:
\begin{eqnarray*}
\pi_{1,v} & = & \frac{ \overline{u}t_1 - t_2}{\overline{u}-u } \\
\overline{\pi}_{1,v} & = &  \frac{ -{u } t_1 + t_2}{\overline{u}-u }
\end{eqnarray*}
Substituting in the  equation
\[
   \pi_{1,v} \overline{\pi}_{1,v} = \det ( \rho_{1} ( \sigma_v ) ) =
   \det ( \rho_{2} ( \sigma_v ) )
\]
and simplifying, we get
\begin{equation}
  - t_1^2 - t_2^2 + (u + \bar{u} ) t_1 t_2 =(u - \bar{u} )^2d,
\end{equation}
where \( d:= \det ( \rho_{1} (\sigma_v ) )=  \det ( \rho_{2} (\sigma_v ) )\),
since we have assumed that the determinant characters are equal.

Since $u$ generates at most quadratic extension of $F$, and $u\neq \overline{u}$,
it follows that the elements $a=(u + \bar{u} )$ and $ b= (u - \bar{u} )^2\neq 0$ belong to $F$. The above equation becomes

The above equation simplifies as below depending on \( u \):
\begin{equation} \label{notequal}
    t_1^2 + t_2^2 -a t_1 t_2 = bd
\end{equation}

Let $\rho= \rho_{1}\times \rho_{2}: G_K\to  GL_2(F)\times  GL_2(F)$
be the product representation
and $G$ be the algebraic monodromy group corresponding to $\rho$.

Let \( X_u \) be the subvariety of \( GL_2 \times GL_2 \)
defined by Equation (\ref{notequal}).
The variety \( X_u \) is a closed subvariety and closed
under the adjoint action of \( GL_2 \times GL_2 \).

By Theorem \ref{Algebraic-Chebotarev},
there exists a connected component \( G^\phi \) that is contained in
\( X_u \). There is a place $w$ of
$\bar{K}$, such that the
associated Frobenius element $\rho(\sigma_w)$ belongs to
$G^\phi ( \overline{F} ) $.
Since $\rho ( \sigma_w ) \in GL_2 (F ) \times GL_2(F)$,
this element is rational over $F$.  The translate
$\rho(\sigma_w)G^0$ is a connected component of $G$ defined over $F$ and
is isomorphic to $G^{\phi}$. Hence $G^{\phi}$ is defined over $F$.

By Lemma (\ref{surj-to-SL2}),
the image of the induced map from \( G^0 ( F) \) to
\( GL_2 ( F) \) by the first projection contains \( SL_2 (F ) \).
Hence the image of $G^{\phi}(F)$ contains the translate $A SL_2 (F )$
where $A= \rho_1(\sigma_w)\in GL_2(F)$. In particular, this means that
the element $t_1\in F$ can be an arbitrary element of $F$. Hence,
Equation (\ref{notequal}) continues to have rational solutions $t_2\in
F$ for any element $t_1\in F$.

Considering Equation (\ref{notequal}) as a quadratic equation in
$t_2$, it follows that the discriminant
\[ (at_1)^2-4(t_1^2-bd)= 4bd+(a^2-4)t_1^2\]
takes square values in $F$ for any $t_1\in F$. Here $4bd\neq 0$.
Since $a=u+\overline{u}$ is a sum of roots of unity
and $u\neq \overline{u}$  by embedding $F$ inside $\C$, we conclude
that $a^2\neq 4$.
By Lemma \ref{arith}, this is not possible. This proves Theorem
\ref{poteq}.
\qed

\begin{remark}
The above argument consisting of specializing to elements in the algebraic monodromy
group of the Galois representation to arrive at a suitable contradiction, can be considered as an
argument involving $\ell$-adic analytic continuation of the Galois monodromy. This further explains the earlier
Remark (\ref{Unitary-Monodromy-at-Infinite-place}).
\end{remark}

\section{Proof of Theorem \ref{Mult-one-FF}}
In this section we prove Theorem \ref{Mult-one-FF}. We first recall
some facts about ordinary and supersingular elliptic curves
over finite fields.

\subsection{Ordinary and supersingular reduction}

Let $E$ be an elliptic curve over  a finite field $k$ with \( q = p^n \) elements.  The curve $E$ is said to be \emph{supersingular} if the group $E[p^r]$ of $p^r$-torsion points is $\{ 0 \}$, and is defined to be \emph{ordinary} otherwise (\cite[Chapter V, Section 3]{Si}). It is known that $E$ being ordinary is equivalent to  $a(E,k)$ being coprime to $p$. The Weil bound implies that \( F(E, k) \) is either
\( \Q \) or an imaginary quadratic field.

Define the \emph{Frobenius field} \( F(E, k) \) of \( E \) over \( k \) as  the splitting field of the
characteristic polynomial of the Frobenius endomorphism $x\mapsto x^q$ of $E$ acting on $V_{\ell}(E)$.

\begin{prop}\label{iso-class}
Let $E$ be an elliptic curve over  a finite field $k$ with \( q = p^n \) elements.
\begin{enumerate}
\item If \( E \) is  ordinary, then \( F(E, k) \) is an imaginary quadratic field in which \( p \) splits  completely. Further, $F(E,k)={\rm End}(E)\otimes \Q$. \\
\item If $E$ is super singular over $\mathbb{F}_p$ and $p\geq 5$, then $a(E, \mathbb{F}_p )= 0 $. Hence, $F(E, \mathbb{F}_p ) = \Q(\sqrt{-p}) $ and $p$ ramifies in $\Q(\sqrt{-p})$.
\end{enumerate}

\end{prop}

\begin{proof}
Most of this proposition is proved in (\cite[Chapter V, Section
3]{Si}), except perhaps the fact that for an ordinary elliptic curve
$p$ splits completely in $F(E,k)$ (\cite{W}). To see this,
we observe that there is a faithful
morphism (\cite[page 139]{Si}),
\[ {\rm End}(E) \to {\rm End}(T_p(E))\otimes \Q_p\simeq \Q_p.\]
This implies that  ${\rm End}(E)$ is commutative. Tensoring with
$\Q_p$ yields a (unital) homomorphism of $\Q_p$-algebras
$ {\rm End}(E)\otimes \Q_p \to \Q_p$. Hence $ {\rm End}(E)\otimes
\Q_p$ cannot be a field and this implies $p$ splits
completely in  $ {\rm End}(E)\otimes \Q_p= F(E,k)$. For the second statement, note that since $E$ is supersingular
$p | a(E, \FF_p ) $. The Hasse bound $| a(E, \FF_p) | \leq  2 \sqrt{p} $ together with $ p \geq 5$ implies that
$a(E, \FF_p )= 0 $.
\end{proof}

\begin{remark}
It can be seen  (\cite[Theorem 4.1]{W}) that Part (1) of the foregoing proposition characterizes
ordinary elliptic curves over any finite field $k$.
\end{remark}

\subsection{Image of Galois} In (\cite{S1, S2}), Serre initiated the study
of  the image $\rho_{E,\ell}(G_K)$ of the Galois group and proved the following theorem:

\begin{theorem} (Serre) \label{galoisimage}
Let \( E \) be an elliptic curve over a number field \( K \). Let \( \ell \) be a prime.
Let \( \rho_{E,\ell} \) be the galois representation attached to \( E \). Let $G$ be the Zariski closure
in $GL_2$ over $\Q_\ell$ of the image of the Galois group $\rho_{E,\ell}(G_K)$.
If $E$ does not have complex multiplication, then $G=GL_2$.
\end{theorem}

Given a number field $K$, the set of finite places of $K$ of degree one over $\Q$ is of density
 one. Hence in working with a set of places of  positive upper
density, we can restrict to the subset of places of degree $1$ over
$\Q$. We have the following proposition due to Serre
(\cite[Chapter IV, Exercises,  pages 13-14]{S1}):
\begin{corollary} \label{ssdensity}
Let \( E \) be an elliptic curve over a number field \( K
\) without complex multiplication. Then, the set of places \( v \in
\Sigma_K \) such that \( E \) has  supersingular reduction at \( v \)
has upper density \( 0 \).
\end{corollary}
\begin{proof} Since $E$ does not have CM, the Zariski closure of the
image of Galois is $GL_2$. At a place $v$ of
 degree one over $\Q$ having supersingular reduction for $E$,
$a_v(E)=0$ provided $Nv = p \geq
5$. Since the set $X=\{g\in GL_2\mid {\rm Trace}(g)=0\}$ is a  proper
closed conjugation invariant subset of $GL_2$, the proposition follows
from  Theorem \ref{Algebraic-Chebotarev}.
\end{proof}

\subsection{Proof of Theorem \ref{Mult-one-FF}}
Suppose \( E_1 \) and \( E_2 \)
are isogenous over a finite extension \( L \) of \( K \). Consider the
curves over $L$.
For any place $w$ of $L$ where both the elliptic curves have good
reduction, the reduced curves $E_{1,w}$ and $E_{2,w}$ are
isogenous. Hence the characteristic polynomial of the Frobenius
conjugacy classes are equal and their associated Frobenius fields
$F(E_1, w)$ and $F(E_2,w)$ are isomorphic.

If $w$ is a place of $L$ of degree one over $K$, then $E_{1,w}$ is
isomorphic to $E_{1,v}$ and hence they have the same Frobenius
fields. This holds for $E_2$ as well. Since the set of places $v$ of
$K$ for which there exists a place $w$ of $L$ of degree one over $K$
is of positive density in $K$, it follows that \( S( E_1, E_2) \) has positive
density and hence positive upper density.

We now prove the converse.  Suppose that the upper density of
\( S :=S (E_1, E_2 ) \) is positive.  Since \( E_1 \) is without complex
multiplication, by Proposition \ref{ssdensity},
the set of places \( v \in \Sigma_K \) such that \(E_{1v} \)
is ordinary has density \( 1 \).  Let \
\[
S_1 := \{ v \in S ~|~ E_{1,v} ~\textrm{is~ordinary and} ~{\rm
  deg}_{\Q}(v)=1 \}.\]
 Thus, \( ud
(S_1 ) = ud(S) >0\).

By Proposition \ref{iso-class}, \( E \) has good
ordinary reduction at \( v \)  if and only if \( F(E,v)\)
is an imaginary quadratic field and \( p_v \) splits in \( F(E, v )\),
where \( p_v \) is the prime of \( \Q \) that lies below \( v
\). This implies that \( p_v \) splits in \( F(v)=F(E_1, v )  = F( E_2, v )
\). Consequently, every \( v \in S_1 \) is a place of good
ordinary reduction for both \( E_1 \) and \( E_2 \).

For \( v \in S_1\), let \( \pi_{1,v} , \overline{\pi}_{1,v} \) and  \( \pi_{2,v} ,
\overline{\pi}_{2,v} \) be respectively the roots of the
characteristic polynomials \( \phi_v (E_1, t) \) and \(  \phi_v
(E_2,t) \). Thus,
\[ \pi_{1,v} , \overline{\pi}_{1,v}  = \pi_{2,v} \overline{\pi}_{2,v}= p_v\]
As ideals of \( F(v) := F(E_1,v) = F(E_2 ,v) \), we have:
\[ ( \pi_{1,v}) ( \overline{\pi}_{1,v} ) = ( \pi_{2,v} ) (
\overline{\pi}_{2,v} ) = ( p_v ).
\]
By unique factorization theorem for ideals, it follows that
\(\pi_{1,v} = u \pi_{2,v}\) or  \( \pi_{1,v} = u
\overline{\pi}_{2,v}\),
where \( u \) depends on \( v \in S_1\) and is a unit of \( F(v)\).
Renaming if need be, one can assume that
\begin{equation}
 \pi_{1,v} = u \pi_{2,v} .
\end{equation}
Since the units in $F(v)$ are roots of unity, it follows that the the
representations $\rho_{E_1,\ell}$ and $\rho_{E_2,\ell}$
are locally potentially equivalent:
\begin{equation}\label{12thpower}
 \rho_{E_1,\ell}(\sigma_v)^{12}= \rho_{E_2,\ell}(\sigma_v)^{12},
\end{equation}
for $v\in S_1$.

Since $E_1$ is assumed to be non-CM, by Theorem \ref{galoisimage}, the
Galois monodromy group $G_1=GL_2$. Hence by Theorem \ref{poteq}, the
representations  $\rho_{E_1,\ell}$ and $\rho_{E_2,\ell}$ are
potentially equivalent.

By Faltings theorem, it follows that $E_1$ and $E_2$ are isogenous
over a finite extension of $K$. This proves Theorem \ref{Mult-one-FF}.
\qed

\begin{remark}\label{Serre-proof}
J. -P. Serre pointed out an error in an earlier version of this paper,
and presented an alternate direct Galois
argument. We quote his argument: assume that both curves are non CM, and
non isogenous (over any extension of  K).
The Galois group acting on their ${\ell}$-division points is, for
$\ell$  large, the subgroup  $H_{\ell}$  of  $GL(2, \FF_\ell ) \times GL(2, \FF_\ell ) $ made up of the pairs having
the same determinant.

Let  $H'_\ell$  be the subset of  $H_\ell $  made up
of the pairs  $(g,g')$  where  $g,~g'$  are semi-simple, with distinct
eigenvalues, and their eigenvalues either are both in  $\FF_\ell$, or are not
in  $\FF_\ell$. This set is characterized by the discriminants
of the characteristic polynomials of $g$ and $g'$ either being
simultaneously squares or non-squares in $\FF_{\ell}$, and define
proper open subsets of $H_{\ell}$. Hence the Haar measure of
$H'_{\ell}$ is strictly less than $1$.

Suppose that the Frobenius eigenvalues of $E_1$ and $E_2$ at a place $v$ differ by a root of unity
$\zeta_k$ as in Equation (\ref{12thpower}). Assume that $\ell$ splits completely in $\Q(\zeta_{12})$. We have
$ ( \rho_{E_1,\ell} (\sigma_v), \rho_{E_2,\ell} (\sigma_v) ) \in H'_{\ell}$.

By Chebotarev density
theorem, the density $\alpha_{\ell}$ of such
primes for a chosen $\ell$ is less than $1$. Upon considering a
sufficiently large collection $P$ of rational primes $\ell$ as above, the
density of the set of places $v$ satisfying  Equation
(\ref{12thpower}) is less than $\prod_{\ell\in P}\alpha_{\ell}$ and
this product goes to zero as $P$ becomes large. This proves Theorem
\ref{Mult-one-FF}.
\end{remark}

\begin{remark}
Serre further observes the following: A computation shows that the density of  $H'_{\ell}$ in $H_{\ell}$  is $1/2 +
O(1/l)$. If one takes $ n$  different primes, one gets a density close to
$(1/2)^n$. By Chebotarev, this implies the result.

One advantage of the method is that it allows a quantitative result: the number of primes with norm  $< X$  such that the corresponding imaginary quadratic fields are the same is slightly smaller than
$X/ \log X$, at least under GRH. This can be done by a sieve argument a la Selberg, but it is a bit
complicated.
\end{remark}

\begin{remark}
The asymptotic behaviour of the set of places \( v \) for which the
associated Frobenius field is a given imaginary quadratic  field \( F
\) has been studied by various authors. Lang and Trotter in 1976
\cite{LT} conjectured:

\begin{conjecture}\label{LT-classical}
If \( E \) is an elliptic  curve defined over the field of rational numbers without
complex multiplication and \( F \) an imaginary quadratic field, then, as \( x \rightarrow \infty \),
\[
S(x, E, F) := \# \{ p \leq x ~\mid ~ F(E,p) = F \} \sim C(E, F) \frac{x^{1/2} }{\log x}
\]
for some positive constant \( C (E,F ) \) depending on \( E \) and
\( F \).
\end{conjecture}

Conjecture (\ref{LT-classical}) has been extensively studied by many authors including \cite{CD1, CD2} where it is studied in the
context of elliptic curves and Drinfeld modules.

It is natural to ask a related but different question: How often the Frobenius fields of two elliptic curves coincide? In fact, based on heuristics, the following conjecture is
suggested on page 38 \cite{LT} for non-CM elliptic curves over the rationals. A generalized version is stated below.

\begin{conjecture}\label{LT-isogeny}
Let \( E_1 \) and \( E_2 \) be two elliptic curves over the rationals without complex multiplication. Then,
\( E_1 \) is not isogenous to \( E_2 \) if and only if
\[
S (x, E_1 , E_2 ) := \# \{ p \leq x ~\mid ~ F(E_1,p) = F (E_2, p) \} = O ( \sqrt{x}/ \log x ) .
\]
\end{conjecture}

Theorem (\ref{Mult-one-FF}) proves a weaker version of Conjecture (\ref{LT-isogeny}).
\end{remark}

\begin{remark}
In \cite{A}, an  algorithm is presented to decide when two abelian
varieities are isogenous, and also to  detect elliptic curves with CM.
\end{remark}

\begin{remark}
In the above theorem, it is necessary to assume that at least one of
the elliptic curves is without complex multiplication and can be seen as follows:

Suppose $v$ is a prime of $K$ of degree one over a rational prime $p\geq 5$, at which an elliptic
curve $E$ has good supersingular reduction. The Frobenius field $F(E,v)$ is $\Q(\sqrt{-p})$.
Let $F_1$ and $F_2$ be non-isomorphic imaginary quadratic fields of class number
one. Let $E_1$ and $E_2$ be  CM elliptic curves over $\Q$  with complex
multiplication by $F_1$ and $F_2$
respectively. At the set of primes $p$ of $\Q$ of good reduction for
 $E_1$ and $E_2$, and such that $p$ is inert in both
$F_1$ and $F_2$, the curves $E_1$ and $E_2$ have supsersingular
reduction. Hence there is a set of places of positive density (in fact having density \( \frac{1}{4} \)) at which
the Frobenius fields are isomorphic, but $E_1$ and $E_2$ are
non-isogenous.
\end{remark}

\begin{remark}
It can be seen that we can modify and prove the theorem under the assumption
that the upper density of the set of finite places \( v \) of \( K \)  for which
both the elliptic curves have good ordinary reduction at \( v \) is positive.
\end{remark}

\section{Proof  of Theorem \ref{FF-with-CM}}
Suppose \( E \) has complex multiplication by an imaginary quadratic
field \( F \). We want to show that
 the set \( S(E,F) := \{ v \in \Sigma_K ~| ~
F(E,v) = F \} \) has positive upper density.

Let \( v \)  be a place of \( K \) of good reduction for \( E \)
with CM by \( F \). From Proposition (\ref{iso-class}), the following
can be seen to be equivalent:

\begin{enumerate}
\item  \( E \) has ordinary reduction modulo \( v \).
\item  \( F(E, v ) = F \).
\item \( p_v \) splits in \( F \), where
\( p_v \) denotes the rational prime of \( \Q \) that lies below \( v \).
\end{enumerate}

Let \( L \) be the compositum of $K$ 
and \( F \).
Let \( Spl(L/\Q ) \) be the set of all primes \( p \)
that split completely in \( L \). Let
\[
 S := \{ v \in \Sigma_K ~|~ v \textrm{~lies~over}~ p \in Spl(L/\Q ) \}.
\] Thus, for a finite place \( v \in S \), \( \deg v \) is \( 1
\). Since every prime \( p \in  Spl(L/\Q ) \) also splits in \( F \),
it follows that \( F(E, v ) = F \)  for \( v \in S \). By the very
construction, \( S \subseteq S(E,F) \). Since every place \( v \in S
\) is of degree \( 1 \) and lies over the primes of \( Spl(L/ \Q ) \),
\[
 ud ( S(E,F) ) \geq ud (S) \geq ud( Spl(L/\Q ) ) = \frac{1}{[L: \Q]} > 0.
\]
\medskip\par
In the converse direction, we want to prove that if for some
imaginary quadratic field \( F \), \( ud (S(E, F) ) > 0 \),
then \( E \) has complex multiplication by \( F \). Without affecting the density, we will assume
that the places in $S(E, F)$ are of degree one over $\Q$ with residue
characteristic at least $5$.
\medskip\par

\noindent {\bf Case 1}:  Suppose \( E \) has complex multiplication
by an imaginary quadratic field \( F^\prime = \Q ( \sqrt{-d} ) \).
We want to prove that
\( F^\prime = F \). Let \( S = S( E, F) \). We can assume after
removing a finite set of places from $S$ that for \( v \in S \),
\( E \) has good reduction modulo \( v \) and \( p_v \) is not ramified in \(
F^\prime \).

Suppose \( p_v \) is inert in \( F^\prime \).  By Proposition \ref{iso-class}, \( E
\) has supersingular reduction
modulo \( v \) and \( F(E, v ) = \Q ( \sqrt{-p_v} )  = F \). The set
of such $v$ is finite.

Hence  for some \( v \in S \), \( p_v \) splits in \( F^\prime \).
This implies that the Frobenius field at \( v \) equals the
CM field, i.e. \(  F(E, v ) = F^\prime \). On the other hand, since \(
v \in S = S (E, F) \), we have \( F(E,v) = F \). This proves \( F =
F^\prime \).

\medskip\par

\noindent {\bf Case 2}:  Let us now consider the case when \( E \) is an elliptic curve over \( K \) without complex
multiplication. The idea is to
construct an elliptic curve, say \( E^\prime \), over a suitable number field with complex multiplication by \( F \) and to
apply Theorem (\ref{Mult-one-FF}) to prove that \( E \) and \( E^\prime \) are isogenous over some extension of \( K \).

\medskip\par

Let \( \mathcal{O}_F \) be the ring of integers of \( F \). Let \(
E^\prime \) be the elliptic curve over \( \C \) such that \( E^\prime
( \C ) \simeq \C / \mathcal{O}_F \). The theory of complex
multiplication implies that \( E^\prime \) is defined over \( H :=
H(F) \), the Hilbert class field of \( F \). Let \( L := HK \) be the
compositum of \( H \) and \( K \).

We wish to apply Theorem
(\ref{Mult-one-FF}) to the two elliptic curves \( E \) and \( E^\prime
\) considered as elliptic curves defined over \( L \). Thus, we need
to prove that the set of places \( w \) of \( L \) such that \( F(E,w)
= F(E^\prime,w) \) has positive upper density.

\medskip\par
Let us denote by \( S_K \) the set of degree \( 1 \) places \( v \in S(E, F) \subseteq \Sigma_K \). Then,
\( ud ( S_K ) = ud ( S(E,F)) \).
Let \( S_\Q \) be the set of primes \( p_v \) of \( \Sigma_Q \) that
lie below the places of \( v \in S_K \). Then \( ud ( S_\Q ) \)
is also positive.

Let \( p \in S_\Q \) and let \( v \) be a place of \( K \) that lies above \( p = p_v \).
By construction, the
Frobenius field at \( v \) equals \( \Q ( \pi_v ) = F \).
Since, \( \pi_v \overline{\pi}_v  = Nv = p = p_v \), the
primes \( p \in S_\Q \) split in \( F \).

Let \( S_F \) be the set of places of \( F \) that
lie over the set of places of \( S_\Q \).
Then \( S_F := \bigcup_{v \in S_K } \{ ( \pi_v ) , ( \overline{\pi}_v ) \} \).
Thus,
\( ud ( S_F ) \) is positive.

The prime ideals of \( S_F \) are principal. By class field theory, they split
completely in the Hilbert class field \( H \) of \( F \).
This implies that the primes \( p \in S_\Q \) split completely in
\( H \).

 Let \( S_L \) be the set of primes of \( L \) that lie above
\( S_K \). Let \( w \in S_L \) be a place above \( v \in S_K \).
Since \( p_v \) splits completely in \( H \), it is easy to see
that the prime \( v \) of \( K \) splits completely in \( L \).
This implies that \( \deg (w) = \deg (v) = 1 \),
implying \( ud ( S_L ) > 0 \).

By considering \( E \) as an elliptic curve over \( L \), it follows that \( F(E,w) = F(E,v) = F \) where \( w \in S_L \) and \(
v \in S_K \) that lies below \( w \).
Similarly, we have \( F( E^\prime, w) = F \).

Applying Theorem (\ref{Mult-one-FF}) to \( E \) and \( E^\prime \) considered
as elliptic curves over \( L \), it follows
that \( E \) and \( E^\prime \) are isogenous
over some finite extension of \( L \), proving the theorem. \qed


\proof[Acknowledgements] 
The results presented in this paper are motivated by a question raised by Dipendra Prasad. We would like to thank Dipendra Prasad for some useful discussions. We gratefully acknowledge the valuable feedback from J.-P. Serre, who pointed out an error in an earlier version of this paper. The second author thanks the School of Mathematics, Tata Institute of Fundamental Research, Mumbai for its excellent hospitality and work environment. The second author also thanks his past employer, International Institute of Information Technology Bangalore, Bangalore, where part of this work was carried out.


\begin{thebibliography}{100}

\bibitem{A} J. D. Achter, Detecting complex multiplication,
Computational Aspects of Algebraic Curves, Lecture Notes Series on Computing, Vol. 13 (2005), World Scientific, 38--50.

\bibitem{B} A. Borel, Linear Algebraic Groups, 2nd enl. ed., GTM 126, Springer-verlag, New York.

\bibitem{CD1} A. Cojocaru and C. David,
Frobenius fields for elliptic curves,
Amer. J. Math. 130 (2008), no. 6, 1535--1560.

\bibitem{CD2} A. Cojocaru and C. David, Frobenius fields of Drinfeld modules of rank \( 2 \), Compositio Math. 144 (2008) 827--848.



\bibitem{LT} S. Lang, H. Trotter, Frobenius distributions in \( GL_2 \) extensions, Lecture Notes in Mathematics, Vol. 504, Springer-Verlag, Berlin-New York, 1976.

\bibitem{PR} Vijay M. Patankar, C. S. Rajan, Locally potentially equivalent Galois representations, J. of Ramanujan Math. Soc.,
Vol. 27, No. 1, (2012), 77--90.

\bibitem{R1} C. S. Rajan, On strong multiplicity one for \( \ell \)-adic representations, Int. Math. Res. Notices, No. 3, (1998),
161--172.

\bibitem{R2} C. S. Rajan, {\em Recovering modular forms and
representations from tensor and symmetric powers},
in Algebra and Number Theory, Proceedings of the Silver Jubilee
Conference, University of Hyderabad, ed. Rajat Tandon, Hindustan Book Agency
(2005), 281-298.

\bibitem{S1} J.-P. Serre, Abelian \( \ell \)-adic representations and elliptic curves, Research Notes in Mathematics, Vol. 7, A.
K. Peters (Eds), Wellesley, Massachusetts, 1998.

\bibitem{S2} J.-P. Serre, Propri\'et\'es galoisiennes des points d'ordre fini des courbes elliptiques, Inv. Math., Vol. 15,
(1972), p. 259--331.

\bibitem{S3} J.-P. Serre, {\em Quelques applications du
th\'{e}or\`{e}me de densit\'{e} de Chebotarev}, Inst. Hautes
\'{E}tudes Sc. Publ. Math. {\bf 54} (1981) 323-4-1.

\bibitem{Si} J. H. Silverman, {\em The arithmetic of elliptic
    curves}, Grad. Texts in Math., {\bf 106}, Springer.

\bibitem{W} W. C. Waterhouse, Abelian varieties over finite fields, Ann. Sci. \'Ecole  Norm. Sup., Vol. 2, (1969), 521--560.



\end{thebibliography}
\end{document}